\documentclass[smallextended]{svjour3}   
\usepackage{amsmath,amssymb,xcolor, enumerate}
\usepackage{hyperref}
\usepackage{subcaption, graphicx}
\usepackage{tikz}
\usetikzlibrary{arrows, calc, shapes}
\usepackage{standalone}
\usepackage{hyperref}

\mathchardef\mhyphen="2D

\newcommand{\conv}{\operatorname{conv}}
\newcommand{\cone}{\operatorname{cone}}
\newcommand{\rank}{\operatorname{rank}}

\newcommand{\pr}{\operatorname{Pr}}
\newcommand{\R}{\mathbb{R}}
\newcommand{\Z}{\mathbb{Z}}

\newcommand{\cA}{\mathcal{A}}

\newcommand{\Po}{\operatorname{P}}
\newcommand{\IP}{\operatorname{IP}}

\newcommand{\WMIP}{{\rm W\!\mhyphen MIP}}
\newcommand{\cG}{\mathcal{G}}

\newcommand{\smallBox}{\ensuremath{B}}

\usepackage{calc}

\DeclareMathAlphabet{\mymathbb}{U}{BOONDOX-ds}{m}{n}

\spnewtheorem{const}{Construction}{\bfseries}{\itshape}

\newcommand{\EQqed}{\tag*{$\qed$}}

\title{The Integrality Number of an Integer Program\footnote{This work expands on a recent IPCO paper~\cite{PSW2020} by providing more details on related work, emphasizing an additional application to non-degenerate integer programs, and correcting a mistake inadvertently introduced in the IPCO review process.}\setcounter{footnote}{3}\footnote{This is a preprint of an article published in Mathematical Programming. The final authenticated version is available online at: {\href{https://doi.org/10.1007/s10107-021-01651-0}{doi.org/10.1007/s10107-021-01651-0}}.}}
\date{}

\author{Joseph Paat$^{\ddag}$ \and Miriam Schl\"oter$^{\mathsection}$ \and Robert Weismantel$^{\mathsection}$}

\institute{
{$\ddag$}\enspace Sauder School of Business, University of British Columbia, BC Canada \\
\email{joseph.paat@sauder.ubc.ca}\\
{$\mathsection$}\enspace Department of Mathematics, ETH Z\"urich, Switzerland\\ \email{\{miriam.schloeter, robert.weismantel\}@ifor.math.ethz.ch}
}

\begin{document}
\maketitle
\begin{abstract}
We introduce the \emph{integrality number} of an integer 
program (IP). 
Roughly speaking, the integrality number is the smallest number of integer constraints needed to solve an IP via a mixed integer (MIP) relaxation.
One notable property of this number is its invariance under unimodular transformations of the constraint matrix.
Considering the largest minor $\Delta$ of the constraint matrix, our analysis allows us to make statements of the following form:  there exists a number $\tau(\Delta)$ such that an IP with $n$ many variables and $n + \sqrt{n /\tau(\Delta)}$ many inequality constraints can be solved via a MIP relaxation with fewer than $n$ integer constraints.
From our results it follows that IPs defined by only $n$ constraints can be solved via a MIP relaxation with $O(\sqrt{\Delta})$ many integer constraints.
\end{abstract}

\section{Introduction.}\label{secIntro}
\thispagestyle{plain}
We denote the polyhedron defined by the constraint matrix $A \in \Z^{m\times n} $ and the right hand side $b \in \Z^m$ by
\begin{equation*}
\Po(A,b) := \{x \in \R^n : Ax \le b\}.
\end{equation*}
We always assume $\rank(A) = n$.
%
%
Lenstra's algorithm~\cite{Lenstra1983} can be used to test the feasibility of  $\Po(A,b) \cap \Z^n$ in polynomial time when the dimension $n$ is fixed; see also~\cite{DP2011,RK1987}.
In fact, his algorithm easily extends to compute a vertex of the {\it integer hull of $\Po(A,b)$}, which we denote by
\[
\IP(A,b) :=  \conv\{x \in \Po(A,b) : ~ x \in \Z^n\}.
\]
See the discussion after Lemma~\ref{PPPoints} for this extension.
More generally, Lenstra's algorithm can be used to find a vertex of a mixed integer hull.
The \emph{mixed integer hull of $\Po(A,b)$ corresponding to a matrix $W \in \Z^{k\times n}$} is
\[
\WMIP(A,b) := \conv\{x \in \Po(A,b) : ~ Wx \in \Z^k\}.
\]
Every mixed integer hull $\WMIP(A,b)$ is a relaxation of the integer hull, and the running time of Lenstra's algorithm to find a vertex of $\WMIP(A,b)$ (if it exists) is polynomial when $k$ is fixed. 
Therefore, if $W$ is chosen such that $\WMIP(A,b) = \IP(A,b)$, then the feasibility of $\Po(A,b) \cap \Z^n$ can be tested in polynomial time when $k$ is fixed rather than under the stronger assumption that $n$ is fixed.
To this end, we define the \emph{integrality number of $\Po(A,b)$}:
\[
i(A,b) :=\min\big\{
 k \in \Z_{\geq 0}: 
\exists~W \in \Z^{k \times n}  \text{ such that } \WMIP(A,b) = \IP(A,b)
\big\}. 
\]
An equivalent definition of $i(A,b)$ is that all vertices of $\WMIP(A,b)$ are integral.
The value $i(A,b)$ can be interpreted as the fewest number of integer constraints that one needs to add to $\Po(A,b)$ in order to test the feasibility of $\Po(A,b) \cap \Z^n$.
The goal of this paper is to constructively bound $i(A,b)$.

\subsection{Related work.}

There is a vast collection of work on how to solve integer programs using mixed integer reformulations. 
Recent work by Hildebrand et al.~\cite{HWZ2017} and Cevallos et al.~\cite{CWZ2018} can be used to bound the integrality number for certain combinatorial sets such as the matching polytope, the cut polytope, and the traveling salesman polytope.
In their work the authors consider mixed integer extended formulations, which recast $\IP(A,b)$ as a mixed integer set in dimension $n' \ge n$ with possibly different polyhedral and integrality constraints.
%
Their results imply that if these combinatorial sets are modeled as $\IP(A,b)$ such that the number of polyhedral constraints is polynomial in the input size of the combinatorial problem, then $i(A,b) \in \Omega(n/\log_2(n))$. 
We refer the interested reader to~\cite{CWZ2018} and the references therein for more on $i(A,b)$ for a large collection of combinatorial problems.
We emphasize that in this paper we only relax the integer constraints of $\IP(A,b)$ and not the underlying polyhedral constraints.

There are some simple situations in which we can bound $i(A,b)$ for general choices of $A$ and $b$. 
For instance, we can always upper bound $i(A,b)$ by $n$ by choosing $W = \mathbb{I}^n$, where $\mathbb{I}^n$ is the identity matrix.
Also, $i(A,b) = 0$ if and only if $\Po(A,b) = \IP(A,b)$, i.e., if and only if $\Po(A,b)$ is a perfect formulation; see Chapter 4 in~\cite{CCZ2014} for an introduction to perfect formulations. 
One sufficient condition for $\Po(A,b)$ to be a perfect formulation is based on the determinants of $A$. 
The largest full rank minor of a matrix $C \in \R^{d\times \ell}$ is denoted by
\[
\Delta(C) := \max\{|\det(B)| : B \text{ is a } \text{rank}(C) \times \text{rank}(C) \text{ submatrix of } C\}.
\]
The topic of integer programming with bounded determinants $\Delta(A)$ has been studied extensively in the past years for general polyhedra~\cite{AEGOVW2016,AWZ2017,GC2016,VC2009} as well as for more structured combinatorial problems~\cite{CFHJW2020,CFHW2020,NSZ2019}.
The matrix $A$ is unimodular if and only if $\Delta(A) =1$, in which case $\Po(A,b)$ is a perfect formulation.
For fixed values of $\Delta(A) \ge 2$, one may ask if $i(A,b)$ can be bounded independently of $n$.
However, Cevallos et al.~\cite[Theorem 5.4]{CWZ2018} construct examples of polyhedra $\Po(A,b)$ with $\Delta(A) = 2$ and $i(A,b) \in \Omega(\sqrt{n}/\log{n})$.

Another tool for bounding $i(A,b)$ is the technique of aggregating columns. 
Let $C \in \Z^{\ell \times n}$ and take $A \in \Z^{m\times n}$ to be a submatrix of 
\begin{equation}\label{eqStdForm}
\left( \begin{array}{r} \mathbb{I}^n \\ - \mathbb{I}^n \\ C \\ -C\end{array}\right).
\end{equation}
%
%
Denote the distinct columns of $C$ by $c_1, \ldots, c_t$, where $ t \le n$, and for each $i = 1, \ldots, t$, let $J_i \subseteq \{1, \ldots, n\}$ index the set of columns of $C$ equal to $c_i$. 
A folklore result in integer programming is that
\[
\IP(A,b) = \conv\bigg\{ x \in \Po(A,b) : \sum_{s \in J_i} x_s \in \Z ~~\forall ~ i \in \{1, \ldots, t\} \bigg\} \quad \forall ~ b \in \Z^m.
\]
See Theorem~\ref{thmStdForm} for a proof of this result.
Hence, $i(A,b)$ is always upper bounded by the number of distinct columns of $C$.
For $r, \Delta \in \Z_{\ge 1}$, define
\[
c(r,\Delta) := \max \bigg\{ d\in \Z_{\ge 1} :
\begin{array}{l}
\exists ~ B \in \Z^{r\times d} \text{ with } d \text{ distinct columns},\\[.05 cm]
\rank(B) = r, \text{ and } \Delta(B) \le  \Delta
\end{array}
\bigg\}.
\]
Set $c(0,\Delta) := 1$.
Heller~\cite{H1957} and Glanzer et al.~\cite{GWZ2018} showed that
\begin{equation}\label{eqDistinctCols}
c(r,\Delta) \le 
\begin{cases}
r^2+r + 1& \text{if } \Delta = 1\\
\Delta^{2+\log_2\log_2(\Delta)} \cdot r^2 +1 & \text{if } \Delta \ge 2.
\end{cases}
\end{equation}
Hence, if $A$ is a submatrix of~\eqref{eqStdForm}, then $i(A,b) \le c(\ell,\Delta(C))$. 

The work most related to this paper is by Bader et al.~\cite{BHWZ2017}, who introduce affine TU decompositions.
An affine TU decomposition of $A$ is an equation $A=A_0 + U W$ such that $[A_0^\intercal~W^\intercal]$ is totally unimodular and $U$ is an integral matrix.
It can be shown that $\WMIP(A,b) = \IP(A,b)$, so computing an affine TU decomposition bounds $i(A,b)$.
 The authors in~\cite{BHWZ2017} prove that it is NP-hard to compute an affine TU decomposition, and they describe decompositions for special families of integer programs~\cite[Section 3]{BHWZ2017}; see also Hupp~\cite{H2017}, who investigated computational benefits of affine TU decompositions.
However, Bader et al.\ do not provide a technique for creating affine TU decompositions for general matrices.
%
Furthermore, affine TU decompositions are not robust under unimodular mappings of $A$; see the text following Definition 2 in~\cite{BHWZ2017}.
On the other hand, the integrality number is preserved under unimodular maps; see Lemma~\ref{lemHNF}.
We highlight this invariance property because transforming $\Po(A,b)$ by a unimodular map preserves integer vertices of any mixed integer hull.
%


\subsection{Statement of Results.}

Our first main result is an upper bound for $i(A,b)$ for general $A$ and $b$ using $\Delta(A)$ and $c(r, \Delta(A))$. 
We are not aware of existing results in the literature that bound the integrality number in such generality.

\begin{theorem}\label{thmNearlySquare}
Let $A \in \Z^{m\times n}$.
Suppose
\[
A = \begin{pmatrix}A^1 \\ A^2 \end{pmatrix},
\]
where $r := \rank(A^2)$ and $A^1 \in \Z^{n\times n}$ with $\delta := |\det(A^1)| \ge 1 $.
Then 
\[
i(A,b) \le [4\delta^{1/2}+\log_2(\delta) ] \cdot \min\{c(r, \Delta(A^2)),  c(r, \Delta(A))\} \qquad \forall ~ b \in \Z^m.
\]
\end{theorem}

The matrix $W$ underlying the proof of Theorem~\ref{thmNearlySquare} can be constructed in polynomial time without the need to compute $\Delta(A)$, which is NP-hard even to approximate~\cite{DEFM2015,P2004}.
%
%
%
Theorem~\ref{thmNearlySquare} and inequality~\eqref{eqDistinctCols} show that $i(A,b) \le \tau(\Delta(A)) \cdot r^2$ for some value $\tau(\Delta(A)) \ge 1$. 
Therefore, if $r < \sqrt{n / \tau(\Delta(A))}$, then $i(A,b) < n$. 
We state this consequence as a corollary.
\begin{corollary}
Let $A \in \Z^{m\times n}$.
There exists a number $\tau(\Delta(A)) \ge 1$ that satisfies the following: 
if $m \le n + \sqrt{n/\tau(\Delta(A))}$, then for all $b \in \Z^m$, the feasibility of $\Po(A,b) \cap \Z^n$ can be tested using a $\WMIP$ relaxation with fewer than $n$ integrality constraints.
\end{corollary}

The integrality number $i(A,b)$ is defined as a minimum over all mixed integer hulls $\WMIP(A,b)$ that equal $\IP(A,b)$. 
However, in order to constructively bound $i(A,b)$, we consider a certain family of mixed integer hulls derived by representing groups of variables with a single integer constraint.
%
%
The next result illustrates how grouping variables can bound $i(A,b)$.
For $K^1, K^2 \subseteq \R^n$, define $K^1 + K^2 := \{x+y : x \in K^1\text{, }y \in K^2\}$.
%

\begin{theorem}\label{thmStdForm}
Let $C \in \Z^{\ell \times n}$ and $A \in \Z^{m\times n}$ be any submatrix of the matrix~\eqref{eqStdForm}.
If $B,T \subseteq \Z^\ell$ are such that the set of distinct columns of $C$ is contained in
\(
B + (T \cup \{0\}),
\)
then $i(A,b) \le |B|+|T|$ for all $b \in \Z^m$.
\end{theorem}

Classical variable aggregation can be recovered from Theorem~\ref{thmStdForm} by setting $B$ to be the set of columns of $C$ and $T$ to be the empty set. 
The proof of this theorem is partially inspired by Bader et al.~\cite{BHWZ2017}, who prove the result when $\ell = 1$ using a combinatorial argument. 
The proof of Theorem~\ref{thmStdForm} leads to an interesting problem on covering columns of a matrix with bounded determinants.
We discuss this problem in Section~\ref{secGeometry}, and we believe it could be of independent interest in future research. 
In Section~\ref{secSquare}, we apply the covering results to prove both Theorems~\ref{thmNearlySquare} and~\ref{thmStdForm}.
We conclude in Section~\ref{secOtherIP} by applying these theorems to special families of integer programs. 
We examine non-degenerate integer programs and improve bounds for the integrality number implied by results in~\cite{AEGOVW2016}.
We also consider consequences of our results on the asymptotic properties of optimal solutions to integer programs. 
%

\medskip
\noindent For the remainder of the paper, we set $\Delta := \Delta(A)$ unless stated otherwise.

\medskip
%

\noindent\textbf{Notation and preliminaries.}
Denote the largest minor of a matrix $C \in \R^{m\times n}$ by
\[
\Delta^{\max}(C) := \max\{|\det(B)| : B \text{ a submatrix of } C\}.
\]
A matrix $W \in \Z^{k\times n}$ is \emph{totally unimodular} if $\Delta^{\max}(W) \le 1$.

Denote the $i$-th row of $C$ by $C_i$.
For $I \subseteq \{1,\ldots,m\}$, let $C_I$ denote the $|I| \times n$ matrix consisting of the rows $\{C_i : i \in I\}$.
Denote the $d \times k $ all zero matrix by $\mymathbb{0}^{d \times k}$.
We use $0$ to denote the zero vector in the appropriate dimension.
We view $B \in \R^{d\times p}$ as a matrix and a set of column vectors in $\R^d$.
So, $v \in B$ implies that $v$ is a column of $B$, and $|B|$ denotes the number of distinct columns of $B$.

The following result is used extensively throughout the paper. 
A proof can be found in~\cite[\S VII]{barv2002}.
\begin{lemma}\label{PPPoints}
Let $B \in \Z^{d\times d}$ be an invertible matrix. 
Then
\[
|\{B \lambda \in \Z^d : \lambda \in [0,1)^d\}| = |\det(B)|.
\] 
For every $v \in \Z^d$, there are unique vectors $\lambda \in [0,1)^d$ and $\tau \in \Z^d$ such that $v = B (\lambda+\tau)$.
\end{lemma}

For completeness, we sketch how Lenstra's algorithm can be used to find a vertex of $\IP(A,b)$ when $\IP(A,b) \neq \emptyset$.
Our sketch uses induction on the dimension of $\Po(A,b)$, denoted by $\dim \Po(A,b)$. 
Clearly, Lenstra computes a vertex of $\IP(A,b)$ if $\dim\Po(A,b)=0$. 
If $\dim\Po(A,b) > 0$, then find a vertex $x^*$ of $\Po(A,b)$.
Suppose that $x^*$ uniquely maximizes $x \mapsto c^\intercal x$ over $\Po(A,b)$ for some $c \in \R^n$. 
The vector $c$ is linearly independent of the set of implied equations
\[
A^{=} := \{a \in \R^n\setminus \{0\} : \exists ~ \beta \in \R ~\text{such that}~a^\intercal x = \beta~~\forall ~x \in \Po(A,b)\}
\]
 because $x^*$ uniquely maximizes $x \mapsto c^\intercal x$, and $c$ can be chosen to have an encoding length polynomial in that of $\Po(A,b)$; see~\cite[Chapters 8 and 10]{AS1986}.
Lenstra's algorithm and binary search can find a maximizer $z^*$ of $x \mapsto c^\intercal x$ over $\Po(A,b) \cap \Z^n$.
The polyhedron ${\rm Q} := \{x \in \Po(A,b) : c^\intercal x = c^\intercal z^*\}$ satisfies $\dim Q < \dim \Po(A,b)$ because $c$ is linearly independent of $A^{=}$.
Moreover, $\conv({\rm Q} \cap \Z^n)$ is a face of $\IP(A,b)$, so vertices of $\conv({\rm Q} \cap \Z^n)$ are vertices of $\IP(A,b)$.
Thus, induction can be applied to ${\rm Q}$ to compute a vertex of $\IP(A,b)$.
%

\section{A geometric tool for the analysis of $i(A,b)$.}\label{secGeometry}

The purpose of this section is to solve the following covering problem.

\begin{definition}[Covering problem]\label{defnCCP}
Let $C \subseteq \R^{\ell}$ be finite.
The covering problem is to find sets $B, T \subseteq \R^\ell$ solving
\[
\min\big\{ |B|+|T| : ~ C \subseteq B + (T \cup \{0\}) ~\text{and}~ B, T \subseteq \R^\ell ~\text{are finite}\big\}.
\]
\end{definition}

\begin{figure}
\centering
\begin{tabular}{@{\hskip .1 cm}c@{\hskip .5 cm}c}
\begin{tikzpicture}


	 \foreach \i in {0,1,2,3,4,5}
	{
		\foreach \j  in {0,1,2}
		{
		
			\draw[draw = none, fill = blue!50, opacity = .5] (\i,\j) circle (5 pt);
		}
	}
	\draw[draw = none, fill = blue!50, opacity = .5](0,3) circle (5 pt);
	\draw[draw = none, fill = blue!50, opacity = .5] (6,0) circle (5 pt);

	 \foreach \i in {0,3}
	{
		\foreach \j  in {0,2}
		{
			\draw[line width = 1 pt, dashed] (\i,\j) circle (5 pt);
			\draw[line width = 1 pt, dashed] (\i+1,\j) circle (5 pt);
			\draw[line width = 1 pt, dashed] (\i+2,\j) circle (5 pt);
			\draw[line width = 1 pt, dashed] (\i,\j+1) circle (5 pt);
			\draw[line width = 1 pt, dashed] (\i+1,\j+1) circle (5 pt);
			\draw[line width = 1 pt, dashed] (\i+2,\j+1) circle (5 pt);
		}
	}
	
	\draw[line width = 1 pt, dashed]  (6,0) circle (5 pt);
    \draw[line width = 1 pt, dashed]  (6,2) circle (5 pt);
	
	 \foreach \i in {0,1,2}
	{
		\foreach \j  in {0,1}
		{
			\draw[draw = none, fill = red!75, opacity = .75] (\i,\j) circle (3 pt);
		}
	}
	\draw[draw = none, fill = red!75, opacity = .75] (6,0) circle (3 pt);
	\draw[draw = none, fill = red!75, opacity = .75] (0,3) circle (3 pt);

	\foreach \i in {0,1,2,3,4,5,6}
	{
		\foreach \j  in {0,1,2,3,4}
		{
			\fill (\i,\j) circle (1.5pt);
		}
	}
	\draw[line width = .75 pt,-latex](0,0) -- (0,4.5);
	\draw[line width = .75pt,-latex](0,0) -- (6.75,0);

\end{tikzpicture}
&
\begin{tikzpicture}

	 \foreach \i in {0,1,2,3,4,5}
	{
		\foreach \j  in {0,1,2}
		{
		
			\draw[draw = none, fill = blue!50, opacity = .5] (\i,\j) circle (5 pt);
		}
	}
	\draw[draw = none, fill = blue!50, opacity = .5](0,3) circle (5 pt);
	\draw[draw = none, fill = blue!50, opacity = .5] (6,0) circle (5 pt);
	
	 \foreach \i in {0,3}
	{
		\foreach \j  in {0,2}
		{
			\draw[line width = 1 pt, dashed] (\i,\j) circle (5 pt);
			\draw[line width = 1 pt, dashed] (\i,\j+1) circle (5 pt);
			\draw[line width = 1 pt, dashed] (\i,\j+2) circle (5 pt);
			\draw[line width = 1 pt, dashed] (\i+1,\j+1) circle (5 pt);
			\draw[line width = 1 pt, dashed] (\i+1,\j) circle (5 pt);
		}
	}
	
	\draw[line width = 1 pt, dashed] (5,0) circle (5 pt);
	\draw[line width = 1 pt, dashed] (5,1) circle (5 pt);
	\draw[line width = 1 pt, dashed] (5,2) circle (5 pt);
	\draw[line width = 1 pt, dashed] (6,1) circle (5 pt);
	\draw[line width = 1 pt, dashed] (6,0) circle (5 pt);
		
	\draw[line width = 1 pt, dashed] (2,0) circle (5 pt);
	\draw[line width = 1 pt, dashed] (2,1) circle (5 pt);	
	\draw[line width = 1 pt, dashed] (2,2) circle (5 pt);
	
	\draw[draw = none, fill = red!75, opacity = .75] (0,0) circle (3 pt);
	\draw[draw = none, fill = red!75, opacity = .75] (1,0) circle (3 pt);
	\draw[draw = none, fill = red!75, opacity = .75] (1,1) circle (3 pt);
	\draw[draw = none, fill = red!75, opacity = .75] (0,2) circle (3 pt);
	\draw[draw = none, fill = red!75, opacity = .75] (0,1) circle (3 pt);

	\foreach \i in {0,1,2,3,4,5,6}
	{
		\foreach \j  in {0,1,2,3,4}
		{
			\fill (\i,\j) circle (1.5pt);
		}
	}
	\draw[line width = 1.1pt,-latex](0,0) -- (0,4.5);
	\draw[line width = 1.1pt,-latex](0,0) -- (6.75,0);

\end{tikzpicture}\\[.25 cm]
\begin{tabular}{r@{\hskip .1 cm}c@{\hskip .1 cm}l}
$B$ & $=$ & $ \{[6,0]^\intercal,[0,3]^\intercal\} \cup  (\{0,1,2\}\times \{0,1\}) $\\[.1cm]
$T$ &$=$& $\{[3,0]^\intercal,[0,2]^\intercal,[3,2]^\intercal\}$
\end{tabular}
&
\begin{tabular}{r@{\hskip .1 cm}c@{\hskip .1 cm}l}
$B$ & $=$ & $  \{[0,0]^\intercal, [1,0]^\intercal,[1,1]^\intercal,[0,1]^\intercal,[0,2]^\intercal\}$\\[.1cm]
$T$ &$=$& $\{[2,0]^\intercal,[3,0]^\intercal,[5,0]^\intercal,[0,2]^\intercal,[3,2]^\intercal\}$
\end{tabular}\\[.5 cm]
(a) & (b)
\end{tabular}
\caption{
Two feasible choices of $B$ and $T$ for the covering problem are shown for $C = \{[6,0]^\intercal, [0,3]^\intercal\} \cup (\{0, \ldots, 5\} \times \{0, \ldots, 2\})  \subseteq \Z^2$.
The vectors in $C$ and $B$ are highlighted in blue and red, respectively, and the vectors in $B+(T \cup \{0\})$ are circled with a dashed line.}
\label{Fig:FigureBoxesAndTranslates}
\end{figure}
Figure~\ref{Fig:FigureBoxesAndTranslates} gives examples of $B$ and $T$ for a particular instance of the covering problem.
A trivial upper bound on the covering problem is the number of vectors in $C$. 
We collect this observation for later use. 

\begin{const}\label{constEnum}
Let $C \subseteq \R^{\ell}$ be finite.
The sets $B = C$ and $T = \emptyset$ satisfy the covering problem and $|B|+|T| = |B| = |C|$. 
\end{const}

For an instance of the covering problem, the optimal value $|B|+|T|$ must be at least $|C|^{1/2}$. 
Otherwise,
\[
|B+(T\cup\{0\})| \le |B| \cdot (|T|+1) \le (|B|+|T|)^2 < \big(|C|^{1/2}\big) \vphantom{a}^2 = |C|,
\]
which cannot occur because $C \subseteq B+(T\cup\{0\})$.
The bound of $|C|^{1/2}$ can be attained for certain choices of $C$ as seen by the following construction and Lemma~\ref{Lem:BoxAndTransLatesConstruction2}.
Bader et al.~\cite[Example 8]{BHWZ2017} demonstrated this result for $\ell = 1$.

\begin{const}\label{constSplit}
Let $\Lambda \in \Z^{\ell \times \ell}$ be a lower triangular matrix of the form
\[
\Lambda = 
\begin{pmatrix}
        \alpha_1& \\[-.15cm]
         \vdots  &\ddots &  \\[-.15 cm]
        *&  \cdots &\alpha_{\ell}
\end{pmatrix},
\]
where $\alpha_1, \ldots, \alpha_{\ell} \in \Z_{\ge 2}$ and $0 \le \Lambda_{i,j} \le \alpha_i -1$ for all $1 \le i <j$.
Set $\delta := \prod_{i=1}^{\ell} \alpha_i $, let $k_i \in \{0, \ldots, \alpha_i-1\}$ for each $i \in \{1, \ldots, \ell\}$, and define the sets
\begin{equation}\label{eq:BoxAndTranslatesConstruction2}
\begin{array}{@{\hskip -.15 cm}rl}
& B := \Lambda \cup \big(\{0,\ldots,k_1\} \times \ldots \times \{0,\ldots, k_\ell\}\big) \\[.15 cm]
\text{and } & T := \big\{z \in \Z^\ell : z_i \in \{0, k_i+1, \ldots, \beta_i  (k_i+1)\}~\forall ~i\in\{1, \ldots, \ell\}\big\} {\setminus \{0\}},
\end{array}
\end{equation}
where
 \[
	\beta_{i} := \bigg\lfloor\frac{\alpha_i-1}{k_i + 1} \bigg\rfloor  \quad \forall ~ i \in \{1,\ldots,\ell\}.
\]
\end{const}

Figure~\ref{Fig:FigureBoxesAndTranslates} (a) illustrates Construction~\ref{constSplit} for the following data:
\[
\Lambda = \begin{pmatrix} 6 & 0 \\ 0 & 3 \end{pmatrix}, ~k_1 = 2, ~ \text{and}~k_2 = 1.
\]
%


Next, we show that any choice of $k_1, \ldots, k_{\ell}$ in Construction~\ref{constSplit} yields a feasible solution to the covering problem for subsets of $\Lambda \cup \Psi(\Lambda) $, where
\[
\Psi(\Lambda)  := \{0,\ldots,\alpha_1-1\} \times \ldots \times \{0,\ldots,\alpha_\ell-1\}.
\]
Moreover, there exists a choice of $k_1, \ldots, k_{\ell}$ that attain the $|C|^{1/2}$ bound.

\begin{lemma}\label{Lem:BoxAndTransLatesConstruction2}
Let $\Lambda \in \Z^{\ell \times \ell}$ and $ k_1, \ldots, k_{\ell}$ be as in Construction~\ref{constSplit}, and let $C \subseteq \Lambda \cup \Psi(\Lambda)$.
	The sets $B$ and $T$ defined in~\eqref{eq:BoxAndTranslatesConstruction2} are a feasible solution for the covering problem.
	Also, $k_1, \ldots, k_{\ell}$ can be chosen such that 
	\[
	|B|+|T| \le 4\delta^{1/2}+\log_2(\delta).
	\]
\end{lemma}

 \begin{proof}
  Let $z \in C$. 
  If $z \in \Lambda$, then $z = z+0 \in B+(T \cup \{0\})$. 
  Otherwise, 
  \(
  z \in \Psi(\Lambda).
  \)
   Define the scalars
	\[
	 t_i := \bigg\lfloor\frac{z_i}{k_i+1} \bigg\rfloor \cdot (k_i+1) ~~\text{and}~~ v_i := z_i -t_i \quad \forall ~ i \in \{1,\ldots,\ell\}.
	\]
	By design, $v_i \in \{0,\ldots,k_i\}$ for each $i \in \{1,\ldots,\ell\}$.
	Therefore, $v := (v_1, \ldots, v_\ell)^\intercal \in B$.
	 To show $(t_1,\ldots,t_\ell)^\intercal \in T$, note that $t_i \leq \beta_i\cdot(k_i+1)$ for each $i \in \{1,\ldots, \ell\}$.
	 Hence, $C \subseteq B+(T \cup\{0\})$.

	It is left to choose $k_1, \ldots, k_{\ell}$ such that $|B|+|T| \le 4{\delta}^{1/2}+\log_2({\delta}) $.
	We proceed with two cases. 
	To simplify the presentation, we assume without loss of generality that $\alpha_1 \leq \alpha_2\leq \ldots \leq \alpha_{\ell}$. 
	Note that $|\Psi(\Lambda)| =\prod_{i=1}^\ell \alpha_i = \delta \ge 2^{\ell}$.
		\noindent\textbf{Case 1.} Assume $\alpha_\ell = {\delta}^\tau$ for $\tau \geq 1/2$. 
		Hence,
		\[
		\prod_{i=1}^{\ell-1} \alpha_i   = {\delta}^{1 - \tau} \leq  {\delta}^{1/2}.
		\]
		Set $\sigma :=\tau - 1/2 \geq 0$ and note that $1-\tau + \sigma = 1/2$.
		Define $k_{\ell} := \lceil {\delta}^\sigma \rceil$ and $k_i := \alpha_i-1$ for each $i \in \{1, \ldots, \ell-1\}$.
		The value $\beta_{\ell}$ in~\eqref{eq:BoxAndTranslatesConstruction2} satisfies 
		%
				\[
		\beta_{\ell} 
		= \bigg\lfloor\frac{\alpha_{\ell}-1}{\big\lceil {\delta}^\sigma \big\rceil+1} \bigg\rfloor 
		= \bigg\lfloor\frac{{\delta}^\tau-1}{\big\lceil {\delta}^\sigma \big\rceil+1} \bigg\rfloor 
		\leq 
		\bigg\lfloor\frac{{\delta}^\tau}{\delta^\sigma} \bigg\rfloor 
		=\big\lfloor{\delta}^{1/2}\big\rfloor 
		\le {\delta}^{1/2},
		\]		
		and $\beta_i = 0$ for each $i \in \{1,\ldots, \ell-1\}$.
		{The equation $|T| = \beta_{\ell}$ holds because $0 \not \in T$.}
 		Set 
		\(
		 \overline{\smallBox} := \{0, \ldots, k_1\} \times \ldots \times \{0, \ldots, k_{\ell}\}.
		\)
		A direct computation proves the result in this case: 
  		\begin{align*}
  			|B|+|T| = |\overline{\smallBox}|+|T|+|\Lambda| = &~\prod_{i=1}^\ell(k_i+1) + \beta_\ell +\ell\\
			= & ~ {\delta}^{1-\tau}(\lceil {\delta}^\sigma \rceil+1) + \beta_\ell + \ell\\
			  \le &~\delta^{1-\tau+\sigma}+2\delta^{1-\tau}+\delta^{1/2}+\log_2(\delta)\\
			   \le &~4\delta^{1/2}+\log_2(\delta).
 		\end{align*}

		\noindent\textbf{Case 2.} Assume $\alpha_{\ell} < {\delta}^{1/2} $.
		 This implies that
		 \(
		 {\delta}^{1/2} < \prod_{i=1}^{\ell-1} \alpha_i .
		 \)
		 Let $j \in\{1, \ldots, \ell -2\}$ be the largest index such that 
		 \[
		 \gamma := \prod_{i = 1}^j \alpha_i \leq {\delta}^{1/2}.
		 \] 
		 Let $\sigma \geq 0$ be such that $\gamma \cdot {\delta}^ \sigma = {\delta} ^ {1/2}$.
		 Set $\alpha_{j+1} = {\delta}^\tau$, where $\tau \in (0, 1/2)$. 
		 Note that $0 \leq \sigma < \tau$ and 
		 \[
		 {\delta}^{\tau - \sigma} \cdot \prod_{i = j+2}^\ell \alpha_i = {\delta}^{1/2}.
		 \]
		Define $k_i := 0$ for each $i \in \{1, \ldots, j\}$, $k_{j+1} := \lceil {\delta}^{\tau - \sigma} \rceil$, and $k_i := \alpha_i - 1$ for each $i \in \{j+2, \ldots, \ell\}$.
		The value $\beta_{j+1}$ in~\eqref{eq:BoxAndTranslatesConstruction2} satisfies
		%
		\[
		\beta_{j+1} = \bigg\lfloor\frac{\alpha_{j+1}-1}{\big\lceil {\delta}^{\tau - \sigma} \big\rceil+1} \bigg\rfloor =  \bigg\lfloor\frac{{\delta}^\tau-1}{\big\lceil {\delta}^{\tau - \sigma} \big\rceil+1} \bigg\rfloor \le \bigg\lfloor\frac{{\delta}^\tau}{\delta^{\tau - \sigma} } \bigg\rfloor = \lfloor{\delta}^{\sigma}\rfloor.
		\]
		Hence, 
		\[
		|T| \le \prod_{i=1}^j \alpha_i \cdot \lfloor\delta^{\sigma}\rfloor \le \gamma \cdot \delta^{\sigma} = \delta^{1/2}.
		\]
		Define
		\(
		  \overline{\smallBox} := \{0, \ldots, k_1\} \times \ldots \times \{0, \ldots, k_{\ell}\}.
		\)
	A direct computation proves the result in this case:
		\begin{align*}
			|B|+|T| = |\overline{B}|+|T|+|\Lambda|
			\le &~\prod_{i=1}^\ell(k_i+1) + \delta^{1/2} +\ell\\
			\le &~ ({\delta}^{\tau - \sigma} +2 )\left(\prod_{i = j+2}^\ell \alpha_i\right)  + \delta^{1/2} + \log_2({\delta})\\
			\le &~ 4 \delta^{1/2} + \log_2({\delta}).
			  \EQqed
		\end{align*}
  \end{proof}

If one considers the covering problem for the Cartesian product of two sets $C^1$ and $C^2$, then one can construct a feasible covering by taking the Cartesian product of tuples $(B^1,T^1)$ and $(B^2,T^2)$ that cover $C^1$ and $C^2$, respectively.
The proof is straightforward, and we omit it here.
\begin{lemma}\label{lemProductFormula}
Let $C^1, C^2 \subseteq \R^\ell$.
For each $i \in \{1, 2\}$, let $B^i,T^i \subseteq \R^{\ell}$ satisfy $C^i \subseteq B^i +(T^i\cup\{0\})$. 
Then
\[
C^1\times C^2 \subseteq B^1\times B^2 + (T^1\cup\{0\} \times T^2 \cup \{0\}) \setminus\{0\}.
\]
\end{lemma}
%

\section{From the covering problem to the integrality number.}\label{secSquare}

We turn the geometric constructions from Section~\ref{secGeometry} into algebraic tools for bounding $i(A,b)$.
To this end, consider a matrix $C \in \Z^{\ell\times n}$. 
Suppose that $B,T\subseteq \R^{\ell}$ are finite sets (possibly empty) such that the set of distinct columns of $C$ is contained in $B+ (T  \cup\{0\})$.
Every column $u\in C$ can be written as $u = v+t$ for some $v \in B$ and $t \in T \cup \{0\}$. 
If many representations exist, then choose one.
We can write $C$ as
\begin{equation}\label{eqDefOfW}
C = (B~~T) W, \quad \text{where}~~W:=\begin{pmatrix}  W^B \\ W^T  \end{pmatrix},
\end{equation}
$W^B \in \{0,1\}^{|B| \times n}$, and $W^T \in \{0,1\}^{|T| \times n}$.
%
%
%
\begin{lemma}\label{Lem:Tu}
  	Let $C \in \Z^{\ell \times n}$ and $B, T \subseteq \R^{\ell}$ satisfy $C \subseteq B + (T \cup \{0\})$.
	If $W\in \Z^{k \times n}$ is given as in \eqref{eqDefOfW}, then $W$ is totally unimodular.
\end{lemma}

\begin{proof}
	By Ghouila-Houri~\cite{G1962},\cite[\S 19]{AS1986}, it is enough to show that
	\begin{equation*}
		y := \sum_{w \in \widehat{W} \cap {W}^B} -w ~ + \sum_{w \in \widehat{W} \cap {W}^T} w \in \{-1,0,1\}^{n}
	\end{equation*}
	for every subset $\widehat{W}$ of the rows of $W$.
	Every column $u$ of $C$ can be written as $u = v+t$ for some $v \in B$ and $t \in T \cup\{0\}$. 
	Hence, a column of $W$ has at most two non-zero entries and every non-zero entry equals $1$.
	One of these entries is in the rows of $W^B$ while the other is in $W^T$.
	This shows $y \in \{-1,0,1\}^n$.
	\qed
\end{proof}

The previous result makes use of the covering problem to identify a totally unimodular matrix. 
The next result, which relates to~\cite[Theorem 2]{BHWZ2017}, shows that such a totally unimodular matrix can be used to bound $i(A,b)$.
%
%
%
\begin{lemma}\label{lemOptCond}
Let $C \in \Z^{\ell \times n}$ and $A \in \Z^{m\times n}$ be any submatrix of the matrix~\eqref{eqStdForm}. 
Let $W\in \Z^{k \times n}$ be a totally unimodular matrix such that each row of $C$ is a linear combination of the rows of $W$.
Then 
\(
i(A,b) \le k
\)
for all $b \in \Z^m$.  
\end{lemma}
\begin{proof}
Let $b \in \Z^m$.
If $\WMIP(A,b) = \emptyset$, then $\IP(A,b) = \emptyset$ and hence the result is correct.
Otherwise, choose a vertex $z^*$ of $\WMIP(A,b)$.
It has the form
	\begin{equation}\label{eqVertex}
	z^* = \begin{pmatrix} A_J \\ W \end{pmatrix} ^{-1} \begin{pmatrix} b_J \\ y \end{pmatrix},
	\end{equation}
where $y \in \Z^k$ and $J \subseteq \{1, \ldots, m\}$ satisfies $|J| = n-k$.
The rows of $A_J$ are linearly independent of the rows of $W$. 
By the assumption that each row of $C$ is a linear combination of the rows of $W$, the only rows of $A$ that are linearly independent of $W$ are those not belonging to $C$.
Hence, the rows of $A_J$ are a subset of the identity matrix in~\eqref{eqStdForm}.
Thus, the invertible matrix whose rows are $A_J$ and $W$ is unimodular because $W$ is totally unimodular.
By Cramer's Rule, we have $z^* \in \Z^n$.
Hence, ${i(A,b)} \le k$ because every vertex of $\WMIP(A,b)$ is integral.
\qed
\end{proof}

{We are now prepared to prove Theorem~\ref{thmStdForm}.

\begin{proof}[of Theorem~\ref{thmStdForm}]
By Lemma~\ref{Lem:Tu} and \eqref{eqDefOfW}, there exists a totally unimodular matrix $W$ with $|B|+|T|$ many rows such that each row of $C$ is a linear combination of the rows of $W$. 
Therefore, $i(A,b) \le |B|+|T|$ by Lemma~\ref{lemOptCond}.
\qed
\end{proof}
}

We turn our attention to Theorem~\ref{thmNearlySquare}.
For the remainder of this section, we assume that $A$, $A^1$, and $A^2$ are given as in the theorem.
The first step in our proof is to perform a suitable unimodular transformation to $A$.
The next result states that the integrality number is preserved under unimodular transformations.
Recall $\Delta = \Delta(A)$.

\begin{lemma}\label{lemHNF}
Let $A \in \Z^{m\times n}$, $b \in \Z^m$, and $U \in \Z^{n\times n}$ be a unimodular matrix.
Then $i(A,b) = i(AU,b)$ and $\Delta = \Delta(AU)$.
\end{lemma}

\begin{proof}
Let $W \in \Z^{k\times n}$ be a matrix such that all vertices of $\WMIP(A,b)$ are integral. 
The matrix $U^{-1}$ maps the vertices of $\WMIP(A,b)$ to those of ${WU}\!\mhyphen\text{MIP}(AU,b)$, and $U^{-1}$ maps $\Z^n$ to $\Z^n$.
	Thus, $WU \in \Z^{k\times n}$ and the vertices of ${WU}\!\mhyphen\text{MIP}(AU,b)$ are integral.
	Hence, {$i(A,b) \ge i(AU,b)$}. 
	To see why the reverse inequality holds, it is enough to notice that $U^{-1}$ is unimodular.
	The equation $\Delta = \Delta(AU)$ holds because $|\det(U)| = 1$.
	\qed
\end{proof}

The particular unimodular transformation that we want to apply is the one that transforms $A$ into \emph{Hermite Normal Form} (see, e.g.,~\cite{AS1986}).
%
There exists a unimodular matrix $U$ such that $AU$ is in Hermite Normal Form:
\begin{equation}\label{eqHNF}
\begin{array}{r@{\hskip .5 cm}c@{\hskip .5 cm}l}
\displaystyle {AU} = \begin{pmatrix} A^1 \\[.05 cm] A^2 \end{pmatrix}, & \text{where} & \displaystyle
A^1 = 
\begin{pmatrix}
\mathbb{I}^{n-\ell} ~~ \mymathbb{0}^{(n-\ell) \times \ell}\\[.05 cm]
A_I^1~~
\end{pmatrix}
\\[.675 cm]
&\text{and}&\displaystyle
A_I^1 =    \begin{pmatrix}
        * & ... &* & \alpha_1& \\[-.15cm]
        \vdots& &\vdots & \vdots  &\ddots & & \\[-.15 cm]
        * & ... &*& * & *   &  \alpha_{\ell}
   \end{pmatrix},
   \end{array}
\end{equation}
$\alpha_1, \ldots, \alpha_{\ell} \in \Z_{\ge 2}$, $A^1_{i,j} \le \alpha_{i}-1$ if $j < i$ and  $A^1_{i,j} = 0$ if $j > i$, and $\delta = \prod_{i=1}^{\ell} \alpha_i = \det(A^1)$.
In light of Lemma~\ref{lemHNF}, we assume {$U = \mathbb{I}^n$ and} that $A$ is in Hermite Normal Form for the rest of the section.

We are set to prove $i(A,b)  \le   [4\delta^{1/2}+\log_2(\delta) ] \cdot c(r, \Delta(A^2))$.
We refer to this as Theorem~\ref{thmNearlySquare} {\it Part 1}.

\begin{proof}[of Theorem~\ref{thmNearlySquare} Part 1]
Let $\overline{A}\,^2\in \Z^{r\times n}$ be any submatrix of $A^2$ with $\rank(\overline{A}\,^2) = r$.
There exists a matrix $V^2 \in \R^{(m-n) \times r}$ such that $A^2 = V^2 \overline{A}\,^2$, and the number of distinct columns of $\overline{A}\,^2$ is bounded by $c(r, \Delta(A^2))$.
{Construction~\ref{constEnum} yields sets $\overline{B}\,^2, \overline{T}\,^2 \subseteq \Z^r$ that cover the distinct columns of $\overline{A}\,^2$ and satisfy $ |\overline{B}\,^2|+|\overline{T}\,^2| \le c(r, \Delta(A^2)$.
The sets $B^2 := V^2\overline{B}\,^2$ and $T^2 := V^2 \overline{T}\,^2$ cover $A^2$, and they satisfy $|B^2|+|T^2| \le  |\overline{B}\,^2|+|\overline{T}\,^2| \le c(r, \Delta(A^2))$.}
%
%
Construction~\ref{constSplit} and Lemma~\ref{Lem:BoxAndTransLatesConstruction2} can be applied to $A^1_I$ to build sets $B^I, T^I \subseteq \Z^{\ell}$ {that cover the columns of $A^1_I$ and satisfy $|B^I|+|T^I| \le 4\delta^{1/2}+\log_2(\delta)$.}
%

Let $C$ be the matrix whose rows are the union of the rows of $A^1_I$ and $A^2$.
By Lemma~\ref{lemProductFormula}, we can combine $B^I, T^I$ and $B^2, T^2$ to get a covering of the distinct columns of $C$.
Applying Theorem~\ref{thmStdForm} to $C$ allows us to conclude that
\begin{align*}
i(A,b) \le k & \le |B^I|\cdot|B^2| +  (|T^I|+1)(|T^2|+1) - 1 = |B^I|\cdot|B^2|+|T^I|\\
&\le (|B^I|+|T^I|)\cdot(|B^2|+|T^2|) \le  [4\delta^{1/2}+\log_2(\delta) ] \cdot c(r, \Delta(A^2))
\end{align*}
for every $b \in \Z^m$.
\qed
\end{proof}

In order to prove Theorem~\ref{thmNearlySquare} {\it Part 2}, i.e.,
$i(A,b)  \le   [4\delta^{1/2}+\log_2(\delta) ] \cdot c(r, \Delta)$, we need to overcome the following technical hurdle.
If $r = \rank(A^2)$ is equal to $n$, then $\Delta(A^2)$ is the maximum over all $n\times n$ determinants of $A^2$ and $\Delta(A^2) \le \Delta(A)$.
However, if $r < n$, then $\Delta(A^2)$ may be larger than $\Delta(A)$. 
Below is an example, where $\Delta(A) = 5$ and $\Delta(A^2) = 6$, illustrating this phenomenon:
\[
A = \left(
\begin{array}{ccc}
1 & 0 & 0 \\ 
0 & 1 & 0 \\ 
2 & 4 & 5\\
1 & 4 & 4\\
2 & 2 & 3
\end{array}\right)
~~\text{and}~~
A^2 =  \left(
\begin{array}{ccc}
1 & 4 & 4\\
2 & 2 & 3
\end{array}\right).
\]

We overcome this hurdle in Lemma~\ref{lemGroup}.
We use properties of the parallelepiped generated by the rows of $A^1$, which induces a group of order $\delta$, and group properties that $\Delta(A)$ imposes on $A^1$ and $A^2$ jointly.
The matrix $Y := A^2 (A^1)^{-1}$ acts as a map between these groups of different orders, and we use it to bound how many distinct columns it has.
The matrix $A^2$ is one choice of $YA^1$ in the next lemma, but we state the result in generality.

\begin{lemma}\label{lemGroup}
Let $\Delta$, $\delta$, and $A^1$ be as in Theorem~\ref{thmNearlySquare}.
If $Y \in \R^{r \times n}$ satisfies $\rank(Y) = r$, $\Delta(Y) \le \Delta / \delta$, and $YA^1 \in \Z^{r\times n}$,
then $Y$ has at most $c(r, \Delta)$ many distinct columns.
\end{lemma}

\begin{proof}
Define 
\[
\Pi := \{g \in [0,1)^n : g^\intercal A^1 \in \Z^n \}.
\]
The set $\Pi$ can be viewed as a group whose operation is addition modulo 1. 
This group is isomorphic to $\{g^\intercal A^1 : g \in \Pi\}$, which is the additive quotient group of $\Z^n$ factored by the rows of $A^1$, and the identity element is $0$.
By Lemma~\ref{PPPoints}, we see that $|\Pi| = \delta$ and for all $ z \in \Z^n $, there exists a unique  $g \in \Pi$ and $v \in \Z^n$ such that $z^\intercal = (g + v)^\intercal A^1$.
Hence, there exist matrices $G \in \R^{n \times r}$ and $V \in \Z^{n \times r}$ such that the columns of $G$ are in $\Pi$ and $YA^1 =  (G  +V)^\intercal A^1$.
So, $Y = (G+V)^\intercal$.

The columns $G_1, \ldots, G_r$ of $G$ form a sequence of nested subgroups
\[
\{0\} \subseteq \langle \{G_1\}\rangle \subseteq \ldots \subseteq \langle \{G_1, \ldots, G_r\}\rangle,
\]
where
\(
\langle \Omega \rangle := \{  \sum_{h \in \Omega} \lambda_h h  ~\mathrm{mod}~ 1 :  \lambda \in \Z^{\Omega} \} 
\)
for $\Omega \subseteq \Pi$.
Define $\alpha_1 = |\langle \{G_1\}\rangle|$ and 
\[
\alpha_i = \frac{|\langle \{G_1, \ldots, G_{i}\}\rangle|}{|\langle \{G_1, \ldots, G_{i-1}\}\rangle|}  \quad \forall ~ i \in \{2, \ldots, r\}.
\]
For each $i \in \{1, \ldots, r\}$, $\alpha_i$ is the index of $\langle \{G_1, \ldots, G_{i-1}\}\rangle$ in $\langle\{G_1, \ldots, G_{i}\}\rangle$, and it is equal to the smallest positive integer such that $\alpha_i G_i ~\mathrm{mod}~ 1 \in \langle \{G_1, \ldots, G_{i-1}\}\rangle$.
The fact that $\alpha_1, \ldots, \alpha_r$ are integers is due to Lagrange's Theorem; for more on group theory see~\cite[Chapter 1]{L2000}.
The definition of $\alpha_i$ implies that there exist integers $\beta_{i,1} ,\ldots, \beta_{i,i-1}$ such that 
\[
\alpha_i G_i - \sum_{j=1}^{i-1} \beta_{i,j} G_j 
 \in \Z^n.
\]
We create an invertible lower-triangular matrix $E \in \Z^{r\times r}$ from these linear forms as follows: the $i$-th row of $E$ is $[-\beta_{i,1}, ...,-\beta_{i,i-1}, \alpha_i, 0, ..., 0]$.
By design, $EG^\intercal \in \Z^{r \times n}$ and $E (G+V)^\intercal = E Y\in \Z^{r \times n}$.
Also,  
\[
|\det(E)| = \prod_{i =1}^r \alpha_i =
|\langle \{G_1\}\rangle| \cdot \prod_{i=2}^r \frac{|\langle \{G_1, \ldots, G_{i}\}\rangle|}{|\langle \{G_1, \ldots, G_{i-1}\}\rangle|}=
| \langle\{G_1, \ldots, G_r\}\rangle| \le {\delta},
\]
where the last inequality follows from the fact that $ \langle\{G_1, \ldots, G_r\}\rangle$ is a subgroup of $\Pi$, whose order is $\delta$.
We have $\rank(EY) = \rank(Y) = r$.
Any $r \times r$ submatrix of $EY$ is of the form $EF$ for an $r \times r$ submatrix $F$ of $Y$.
The assumption $\Delta(Y) \le \Delta / \delta$ implies $|\det(F)| \le \Delta/\delta$ and $|\det(EF)| = |\det(E)| \cdot |\det(F)| \le \Delta$.
Hence, $\Delta(EY) \le \Delta$.

Columns of $Y$ are distinct if and only if the corresponding columns of the integer-valued $E Y$ are distinct because $E$ is invertible. 
The function $c(r, \cdot)$ is nondecreasing, so $EY$ has at most $c(r, \Delta(EY)) \le c(r, \Delta)$ many distinct columns.
\qed
\end{proof}
%

\begin{proof}[of Theorem~\ref{thmNearlySquare} Part 2]
Recall the notation $A^1_I$, $A^1$, and $A^2$ in~\eqref{eqHNF}.
By Lemma~\ref{lemOptCond}, in order to bound the integrality number, it suffices to find a totally unimodular matrix $W \in \Z^{k\times n}$ whose row span contains the rows of $A^1_I$ and $A^2$. 
Hence, we may assume without loss of generality that $A^2 \in \Z^{r\times n}$ as opposed to just $\rank(A^2) = r$.

%

The matrix $A^1$ is invertible, so there exist $R \in \R^{r \times (n-\ell)}$ and $Q \in \R^{r \times \ell}$ such that
\(
A^2 = [R ~ Q ] A^1 .
\)
We can write $A$ as
\[
A 
=
\begin{pmatrix}
A^1 \\ 
A^2
\end{pmatrix}
=
\begin{pmatrix}
\mathbb{I}^n \\
R ~ Q 
\end{pmatrix}
A^1 .
\]
Note that 
\begin{equation}\label{eqRisTU}
\Delta^{\max}
\left(
\begin{pmatrix}
\mathbb{I}^n \\
R ~ Q 
\end{pmatrix}
\right)
\le {\frac{\Delta}{\delta}}.
\end{equation}
To see why~\eqref{eqRisTU} is true, take any submatrix $D$ of the matrix in~\eqref{eqRisTU}. 
The matrix in~\eqref{eqRisTU} contains $\mathbb{I}^n$, so we can extend $D$ to a matrix in $\R^{n\times n}$ with the same absolute determinant. 
Hence, $DA^1$ is an $n \times n$ submatrix of $A$ and $|\det(DA^1)| =|\det(D)| \cdot \delta \le \Delta$.

By Lemma~\ref{lemGroup}, $[R ~ Q]$ has at most $c(r, \Delta)$ many distinct columns.
Thus, $[R ~ \mymathbb{0}^{{r}\times \ell}]$ has at most $c(r, \Delta)$ many distinct columns.
We can apply Construction~\ref{constEnum} to obtain sets $B^R, T^R \subseteq \R^r$ that cover $C = [{R} ~ \mymathbb{0}^{{r}\times \ell}]$ and satisfy $ |B^R| +|T^R| \le c(r, \Delta)$.
Construction~\ref{constSplit} and Lemma~\ref{Lem:BoxAndTransLatesConstruction2} can be applied to build sets $B^I, T^I \subseteq \Z^{\ell}$ that cover the columns of $A^1_I$ and satisfy $ |B^I|+|T^I| \le 4\delta^{1/2}+\log_2(\delta)$.
%

{By Lemma~\ref{lemProductFormula}, we can combine $B^R, T^R$ and $B^I, T^I$ into sets $B,T$ that cover the columns of the matrix whose rows are the union of $ A^1_I$ and $[{R} ~ \mymathbb{0}^{{r}\times \ell}]$.
By Lemma~\ref{Lem:Tu} and~\eqref{eqDefOfW}, there exists a totally unimodular matrix $W\in \Z^{k \times n}$, where $k = |B| +|T|$, such that}
\[
\begin{pmatrix} A^1_I \\[.05 cm]
{R} ~ \mymathbb{0}^{{r}\times \ell}\end{pmatrix} = ( B~T) W.
\]
This equation implies that there is a submatrix $V \in \R^{\ell \times k}$ of $(B~T)$ such that $A^1_I = VW$.
Recalling that $A^2 = [R ~ Q ] A^1=   [R~\mymathbb{0}^{r\times \ell}] + QA^1_I$, we see that
\[
C:= \begin{pmatrix} 
A^1_I \\[.05 cm]
A^2\end{pmatrix}
=
\begin{pmatrix} 
A^1_I \\[.05 cm]
 [R~\mymathbb{0}^{r \times \ell}] + QA^1_I
 \end{pmatrix} 
 =
 \left[
(B~T)
 +
\begin{pmatrix} 
\mymathbb{0}^{\ell \times k} \\[.05 cm]
QV
 \end{pmatrix} 
 \right]
 W.
\]
Applying Lemma~\ref{lemOptCond} to $C$ allows us to conclude that
\begin{align*}
{i(A,b)} \le k  &\le |B^I| \cdot |B^R| + (|T^I| + 1) \cdot (|T^R|+1) - 1 = 
 |B^I| \cdot |B^R| + |T^I|\\
 &\le 
 ( |B^I| + |T^I| ) \cdot(|B^R| + |T^R|) 
\le [4{\delta}^{1/2}+\log_2({\delta})] \cdot c({r}, \Delta)
\end{align*}
for every $b \in \Z^m$.
\qed
\end{proof}

\section{Further applications to integer programming.}\label{secOtherIP}

In this section we apply Theorem~\ref{thmNearlySquare} to special families of integer programs. 
First, we discuss non-degenerate integer programs\footnote{The results in this section do not appear in the IPCO version of the paper~\cite{PSW2020}.}.
Next, we consider the setting of asymptotic integer programs\footnote{The results in this section correct a mistake in the IPCO version, where the set defined in~\cite[Theorem 2]{PSW2020} is incorrect. The correct set is given in (12) of that paper as well as~\eqref{eq:DefOfGSquare}.}.

\subsection{Non-degenerate integer programming.}\label{secNondegMatrices}
The matrix $A \in \Z^{m\times n}$ is {\it non-degenerate} if $\det(B) \neq 0$ for every $n\times n$ submatrix $B$ of $A$.
Integer programs with non-degenerate constraint matrices were studied by Artmann et al.~\cite{AEGOVW2016}, who proved that such problems can be solved in polynomial time when $\Delta$ is fixed; see also~\cite{GMP2020,VC2009}. 
These types of integer programs capture, among other things, optimizing over simplices.
Artmann et al.\ argue that non-degenerate matrices adhere to strict properties: either $n \le (2 f(\Delta) + 1)^{\log_2(\Delta)+3} + \log_2(\Delta)$, where $f(\Delta)$ is a number greater than or equal to $\Delta$, or $A$ has at most $n+1$ many rows~\cite[Lemma 7]{AEGOVW2016}.
The next lemma bounds the number of rows, in general.

\begin{lemma}\label{LemnondegRows}
If $A \in \Z^{m\times n}$ is non-degenerate, then $m \le n+\Delta^2$.
\end{lemma}
\begin{proof}
Multiplying $A$ on the right by a unimodular matrix preserves non-degeneracy. 
Hence, assume $A$ is in Hermite Normal Form and recall the notation $A^1, A^2, \alpha_1, \ldots, \alpha_{\ell}$ and $A^1_I$ in~\eqref{eqHNF}.
The Hermite Normal Form depends on the choice of $A^1$.
We may assume $|\det(A^1)| = \Delta$ by choosing $A^1$ to be the $n\times n$ submatrix of $A$ with maximum absolute determinant.
Partition $A_I^1$ as
\[
A_I^1 = [B ~C], ~~\text{where}~~
B := \big(b^1, \ldots, b^{n-\ell}\big) ~\text{and}~C:= \begin{pmatrix}
        \alpha_1& \\[-.15 cm]
        \vdots &\ddots   \\[-.15 cm]
        *& \ldots   &  \alpha_{\ell}
       \end{pmatrix}.
\]

Every row of $A^2$ is a linear combination of the rows of $A^1$, and hence it is of the form 
\(
[\gamma + \lambda^\intercal B ~~ \lambda^\intercal C],
\)
where $\gamma \in \R^{n-\ell}$ and $\lambda \in \R^\ell$. 
Neither $\gamma$ nor $\lambda$ can have a component equal to zero, otherwise there would exist an $n\times n$ submatrix of $A$ whose determinant is equal to zero, contradicting non-degeneracy.
Furthermore, $\|\gamma\|_{\infty}, \|\lambda\|_{\infty} \le 1$, otherwise there would exist an $n\times n$ submatrix of $A$ whose determinant is larger than $\Delta$. 

Fix $\lambda \in [-1,1]^{\ell}$ such that $[\gamma + \lambda^\intercal B ~ \lambda^\intercal C]$ is a row of $A^2$ for some $\gamma \in \R^{n-\ell}$.
We claim that there exists at most one other row of $A^2$ of the form $[\overline{\gamma \vphantom{\lambda}} + \overline{\lambda}\ ^\intercal B ~ \overline{\lambda}\ ^\intercal C]$, where $\overline{\lambda} = \pm \lambda$.
Multiplying a row of $A$ by $-1$ preserves non-degeneracy, so we may assume without loss of generality that $\overline{\lambda} = \lambda$.
If $\ell = n$, then the two rows are equal to $ \lambda^\intercal C$, contradicting that $A$ is non-degenerate. 
If $\ell < n$, then consider the following $n\times n$ submatrices of $A$:
\[
D:=
\begin{pmatrix} 
\mathbb{I}^{n-\ell-1} & \mymathbb{0}^{(n-\ell-1)\times (\ell+1)} \\
\gamma+\lambda^\intercal B & \lambda^\intercal C\\
B & C\end{pmatrix}
=
\begin{pmatrix} 
1 \\[-.25 cm]
& \ddots \\[-.25 cm]
& &1 \\[-.15 cm]
* &\ldots&*&E 
\end{pmatrix}
\]
and
\[
\overline{D}:=\hspace{-.15 cm}
\begin{pmatrix} 
\mathbb{I}^{n-\ell-1} & \mymathbb{0}^{(n-\ell-1)\times (\ell+1)} \\
\overline{\gamma}+{\lambda}^\intercal B & {\lambda}^\intercal C\\
B & C\end{pmatrix}
=
\begin{pmatrix} 
1 \\[-.25 cm]
& \ddots \\[-.25 cm]
& &1 \\[-.15 cm]
* &\ldots&*&\overline{E} \end{pmatrix},
\] 
where
\[
E := 
\begin{pmatrix}
\gamma_{n-\ell} + \lambda^\intercal b^{n-\ell} & \lambda^\intercal C \\[.1 cm] 
b^{n-\ell}  & C
\end{pmatrix}
~~\text{and}~~
\overline{E} := 
\begin{pmatrix}
\overline{\gamma}_{n-\ell} + {\lambda}^\intercal b^{n-\ell} & {\lambda}^\intercal C \\[.1 cm] 
b^{n-\ell}  & C
\end{pmatrix}.
\]
Non-degeneracy of $A$ implies that $E$ and $\overline{E}$ are invertible.
Also, $[\overline{\gamma}_{n-\ell} + {\lambda}^\intercal b^{n-\ell}~{\lambda}^\intercal C]$ can be uniquely written as a linear combination of the rows of $E$ with no coefficient equal to zero.
Thus, 
\begin{equation}\label{eqNonDeg0}
\overline{\gamma}_{n-\ell} \neq {\gamma}_{n-\ell}.
\end{equation}
Assume to the contrary that a third row of $A^2$ has the form $[\tilde{\gamma} + \tilde{\lambda}^\intercal B~\tilde{\lambda}^\intercal C]$, where $\tilde{\lambda} = \pm \lambda$.
Again, we may assume without loss of generality that $\tilde{\lambda} = \lambda$. 
Similarly to above, $[\tilde{\gamma}_{n-\ell} + {\lambda}^\intercal b^{n-\ell}~{\lambda}^\intercal C]$ can be uniquely written as a linear combination of the rows of both $E$ and $\overline{E}$ with no coefficient equal to zero. 
So, 
\begin{equation}\label{eqNonDeg}
\tilde{\gamma}_{n-\ell} \neq \gamma_{n-\ell},~~\text{and}~~
\tilde{\gamma}_{n-\ell} \neq \overline\gamma_{n-\ell}.
\end{equation}
There are exactly two non-zero values $\epsilon \in [-1,1]$ such that $\epsilon + \lambda^\intercal b^{n-\ell} \in \Z$.
This contradicts~\eqref{eqNonDeg0} and~\eqref{eqNonDeg}, which proves the claim.

The number of rows of $A$ is equal to $n$ plus the number of rows of $A^2$.
Every row of $A^2$ is of the form $[\gamma + \lambda^\intercal B ~ \lambda^\intercal C]$, and the previous claim states that at most two rows of $A^2$ correspond to the same $\lambda \in [-1,1]^{\ell}$ up to multiplication by $-1$.
Thus, we complete the proof by showing that, up to multiplication by $-1$, there are at most $(1/2) \cdot \Delta^2$ many choices of $\lambda$ that correspond to a row $[\gamma + \lambda^\intercal B ~ \lambda^\intercal C]$ of $A^2$.
It is sufficient that we count the number of vectors $\lambda \in [-1,1]^{\ell}$, up to multiplication by $-1$, that have non-zero components and satisfy $\lambda^\intercal C \in \Z^{\ell}$. 
{The vectors $\lambda$ that we want to count, along with their negatives, are contained in the set
\[
{\Pi} := \big\{\lambda \in [-1,1]^\ell :\lambda^\intercal C \in \Z^{\ell}~\text{and every component of}~\lambda~\text{is non-zero}\big\}.
\]
Therefore, we complete the proof by establishing $|{\Pi}| \le \Delta^2$.
}

{For each $I \subseteq \{1, \ldots, \ell\}$, consider the half-open unit cube
\[
\begin{array}{rcl}
{K^I}  &:=& \big\{ \lambda \in \R^{\ell} : \lambda_i \in [0,1) ~~\forall ~i \in I~\text{and}~\lambda_i \in (-1,0] ~~\forall ~i \not \in I \big\}
\end{array}
\]
and its topological closure $\overline{K^I}$.
Define the sets
\[
\begin{array}{rcl}
{\Pi^I} &:=& \big\{\lambda \in {K^I} :\lambda^\intercal C \in \Z^{\ell}~\text{and every component of}~\lambda~\text{is non-zero}\big\}~\text{and}\\[.1cm]
\overline{\Pi^I} &:=& \big\{\lambda \in \overline{K^I} :\lambda^\intercal C \in \Z^{\ell}~\text{and every component of}~\lambda~\text{is non-zero}\big\}.
\end{array}
\]
By Lemma~\ref{PPPoints} and the fact that each $\lambda \in \Pi^I$ is non-zero in every component, we see that $|\Pi^{I}| \le \Delta-1$.
The fact also implies that $\overline{\Pi^I} = \Pi^I \cup \{e^I\}$, where $e^I \in \Z^{\ell}$ is defined component-wise to be $e^I_i := 1$ for each $i \in I$ and $e^I_i := -1$ otherwise. 
Hence, $|\overline{\Pi^I}|\le \Delta$ for each $I$.
Observe that $\Pi = \bigcup_{I \subseteq \{1, \ldots, \ell\}} \overline{\Pi^I}$, so 
\[
|\Pi| \le \sum_{I \subseteq \{1, \ldots, \ell\}} \big|\overline{\Pi^I}\big| \le 2^{\ell} \cdot \Delta.
\]
Recall that $\alpha_1, \ldots\alpha_{\ell} \ge 2$.
Therefore, $2^{\ell}\le \prod_{i=1}^\ell \alpha_i = \Delta$ and $|\Pi| \le \Delta^2$. }
%
\qed
\end{proof}

The result of Artmann et al.~\cite{AEGOVW2016} implies that 
\[
i(A,b) \le (2 \Delta + 1)^{\log_2(\Delta)+3} + \log_2(\Delta) \quad \forall ~ b \in \Z^m.
\]
After writing $A$ as in Theorem~\ref{thmNearlySquare}, then we can use Lemma~\ref{LemnondegRows} to bound $r$ by $\Delta^2$ and obtain the following improvement to Artmann et al.
\begin{corollary}
Let $A \in \Z^{m\times n}$ be non-degenerate. 
Then
\[
i(A,b) \le [4\Delta^{1/2}+\log_2(\Delta)] \cdot [\Delta^{6+\log_2\log_2(\Delta)}+1] \quad \forall ~ b \in \Z^m.
\]
\end{corollary}

\subsection{Asymptotic integer programming.}\label{secGenMatrices}

In this section we consider solutions of an integer program close to a vertex of $\Po(A,b)$ for varying choices of $b \in \Z^m$.
The bound on the integrality number $i(A,b)$ provided in Theorem~\ref{thmNearlySquare} may depend quadratically on the number of inequalities in the system $A^2x \le b^2$.
Lemma~\ref{Lemma:Proximity} gives a sufficient condition for when one can remove polyhedral constraints from $A x \le b$ to bound $i(A,b)$.
The proof of the lemma uses the following proximity result due to Paat et al.~\cite{PWW2018}\footnote{The authors in~\cite{PWW2018} consider $\WMIP$s with $W =  [\mathbb{I}^k~ \mymathbb{0}^{n \times k}]$. 
The proof of Lemma~\ref{lemCook} follows nearly verbatim the proof of Theorem 2 in~\cite{PWW2018} by replacing $[\mathbb{I}^k~ \mymathbb{0}^{n \times k}]$ with a general $W$.}.
Set $\Delta^{\max} := \Delta^{\max}(A)$. 

\begin{lemma}[Theorem 2 in~\cite{PWW2018}]\label{lemCook}
Let $b \in \Z^m$ and $W \in \Z^{k\times n}$.
Given any vertex of $z^*$ of $\WMIP(A,b)$, there exists a vertex $x^*$ of $\Po(A,b)$ such that $\|z^* - x^*\|_{\infty} \le k\Delta^{\max}$. 
\end{lemma}

A set $I \subseteq \{1, \ldots, m\}$ is a \emph{basis} if $|I| = n$ and $\rank(A_I) = n$, and $I$ is \emph{feasible} for $\Po(A,b)$ if $(A_I)^{-1}b_I\in \Po(A,b)$.
We say that $A_I \in \Z^{n\times n}$ is a \emph{basis matrix} if $I$ is a basis. 

\begin{lemma}\label{Lemma:Proximity}
    Let $b \in \Z^m$ and $W \in \Z^{k \times n}$.
   Let $I \subseteq\{1,\ldots,m\}$ a feasible basis, $z^*$ be a vertex of $\WMIP(A_I,b_I)$, and $j \in \{1, \ldots, m\} \setminus I$.
    If
    \[
    A_jA_I^{-1} b_I+ (n\Delta^{\max})^2 \le b_j ,
    \]
    then $A_j z^* \le b_j$.
     If the above holds for all $j\in \{1, \ldots, m\} \setminus I$, then $z^* \in \Po(A,b)$.
\end{lemma}

   \begin{proof}
    Let $x^* := A_I^{-1} b_I$ be the unique vertex of $\Po(A,b)$ with respect to the basis $I$.
     Applying Lemma~\ref{lemCook} to $\Po(A_I,b_I)$ and $\WMIP(A_I,b_I)$ shows that $z^*$ satisfies $\|z^* - x^*\|_{\infty} \le k\Delta^{\max}$.
    {Recall that $\Delta^{\max}$ is an upper bound on the absolute value of all determinants of $A$ of any size.
    In particular, $\Delta^{\max}$ is an upper bound on the absolute values of the $1\times 1$ determinants, i.e., on the largest absolute entry $\|A\|_{\infty}$ of $A$.}
	Thus, 
    \[
        |A_j z^*-A_jx^*|  \leq \|A_j\|_1 \cdot \|z^* - x^*\|_\infty \leq (n\|A\|_\infty) \cdot (k\Delta^{\max}) \leq  (n\Delta^{\max})^2.
    \]
    The assumption $A_jA_I^{-1} b_I+ (n\Delta^{\max})^2 \le b_j $ implies that
    \[
    A_j z^*\leq A_jx^* +|A_jz^*-A_jx^*| \leq A_jx^* + (n\Delta^{\max})^2  \le b_j. \qquad \qed
    \]
    %
    %
%
\end{proof}

Each basis matrix $A_I$ is square, so we can apply Theorem~\ref{thmNearlySquare} to represent the integer hull $\IP(A_I,b_I)$ as a mixed integer hull $\WMIP(A_I,b_I)$ with $O(|\det(A_I)|^{1/2})$ many integrality constraints. 
If the conditions of Lemma~\ref{Lemma:Proximity} are met for all $j\in \{1, \ldots, m\} \setminus I$, then we can test feasibility of $\Po(A,b) \cap \Z^n$ using the following algorithm. 
First, find a feasible basis $I$ for $\Po(A,b)$. 
Next, compute a mixed integer hull $\WMIP(A_I,b_I)$ that is equal to $\IP(A_I, b_I)$. Finally, test if $\WMIP(A_I,b_I)$ has a vertex.

It is a strong assumption to suppose Lemma~\ref{Lemma:Proximity} holds for all $j\in \{1, \ldots, m\} \setminus I$.
However, the assumption is met by most choices of $b$.
To formalize the term `most', define the \emph{density of a set $\cA \subseteq \Z^m$} to be
\[
\pr(\cA) := \liminf_{t \to \infty} \frac{|\{-t, \ldots, t\}^m \cap \cA|}{|\{-t, \ldots, t\}^m|}. 
\]
The value $\pr(\cA)$ can be interpreted as the likelihood that the family $\{ \IP(A,b) : b \in \cA\}$ occurs in $ \{ \IP(A,b) : b \in \Z^m\}$.
The functional is not formally a probability measure but rather a lower density function {found in number theory~\cite[Page xii and \S 16]{N2000}}.
Asymptotic integer programs were first considered by Gomory~\cite{G1965} and Wolsey~\cite{W1981}, who showed that the {integer programming} value function is asymptotically periodic. 
This asymptotic setting does not capture every integer program, but as Gomory observed, it can help us understand how $\Delta$ affects $\IP(A,b)$. 
Moreover, certain questions about general integer programs can be reduced to the asymptotic setting.
Wolsey~\cite{W1981} argued that asymptotic integer programs can reduce so-called parametric integer programming to a finite number of cases.
Bruns and Gubeladze demonstrated that the so-called Integer Carath\'{e}odory number can be bounded if certain asymptotic conditions hold true~\cite{BG2004}. 
The functional $\pr(\cdot)$ has also been used to study sparse integer solutions in~\cite{ADOO2017,OPW2020,OPW2019}.

Define the set
\begin{equation}\label{eq:DefOfGSquare}
	\cG := \left\{ b \in \Z^m: 
	\begin{array}{l}
	\text{Either} ~ \Po(A,b) = \emptyset ~\text{or}~A_jA_I^{-1}  b_I + (n\Delta^{\max})^2 \le b_j\\[.1 cm]
	\hspace{.5 in} \forall~ \text{feasible bases } I \text{ and } j \in  \{1,\ldots,m\} \setminus I
		\end{array}\right\}.
\end{equation}
{To motivate the definition of $\cG$, consider the following.
Suppose $\Po(A,b)$ is non-empty and fix a feasible basis $I$. 
If $b_j$ is increased enough for some $j \not \in I$, then the slack $b_j - A_j z$ becomes non-negative for every $z \in \Po(A_I,b_I)$ that is sufficiently close to the vertex $A_I^{-1}b_I$.
In other words, integer vectors $z$ in $\Po(A_I, b_I)$ that are close to $A_I^{-1}b_I$ will satisfy the constraint $A_j z \le b_j$ (see Lemma~\ref{Lemma:Proximity}).
The set $\cG$ contains those right hand side vectors $b$ for which this slack condition is satisfied for all bases $I$ and every $j \not \in I$.
{Following this slack interpretation, it may be evident that $\Pr(\cG) = 1$}; indeed, this is true and we formally argue it below.
In Gomory's work~\cite{G1965}, he analyzes a set $\cG_I$ of right hand sides that are asymptotically deep within the cone $\cone(A_I)$ generated by the columns of $A_I$ for a fixed basis $I$.
Although he does not discuss density, it can be determined that $\cG_I$ has an asymptotic density of one among the integer vectors in $\cone(A_I)$.
The set $\cG$ is the union of $\cG_I$ over all feasible bases $I$, minus common intersection between different $\cG_I$.}

If $b \in \cG$, then either $\Po(A,b)$ is empty, in which case $\IP(A,b)$ is also empty, or testing feasibility of $\IP(A,b)$ can be reduced to testing feasibility of the mixed integer hull $\WMIP(A_I,b_I)$ constructed in Theorem~\ref{thmNearlySquare} for any feasible basis $I$.
{We prove $\Pr(\cG) = 1$ by proving $\Pr(\Z^m \setminus \cG) = 0$.}
We show the latter by demonstrating that $\Z^m \setminus \cG$ is contained in a finite union of hyperplanes in $\Z^m$. 
Given a basis $I$, we set $\Delta_I := |\det(A_I)|$.

\begin{lemma}\label{Lemma:BadPointsOnHyperplanes}
It follows that
\[
	\Z^m\setminus \cG \subseteq\bigcup_{\substack{I \subseteq \{1,\ldots,m\} \\ I ~\mathrm{basis}}} ~ \bigcup_{j \not \in I} ~\bigcup_{r = 0}^{\Delta_I(n\Delta^{\max})^2-1}\{b \in \Z^m : \Delta_I b_j = \Delta_I A_jA_I^{-1} b_I + r\}.
\]
{Furthermore, $|\{-t, \ldots, t\}^m \cap (\Z^m \setminus \cG)| \in O(t^{m-1})$ for each $t \in \Z_{\ge 1}$.}
\end{lemma}

\begin{proof}
	If $b \in \Z^m \setminus \cG$, then there exists a feasible basis $I$ for $\Po(A,b)$ and $j \in  \{1,\ldots,m\} \setminus I$ such that $b_j < A_jA_I^{-1} b_I  + (n\Delta^{\max})^2$.
	Thus,
    \begin{align*}
         \Z^m\setminus \cG \subseteq &\left\{b \in \Z^m:
         \begin{array}{l} 
         \exists \text{ a feasible basis } I \text{ and } j \in \{1,\ldots,m\}\setminus I \\[.05 cm]
         \text{with } b_j < A_jA_I^{-1} b_I +  (n\Delta^{\max})^2
         \end{array}
         \right\}\\[.15 cm]
        =&\left\{b\in \Z^m: 
        \begin{array}{l}
        \exists \text{ a feasible basis } I \text{ and } j \in \{1,\ldots,m\}\setminus I\\[.05 cm]
        \text{with } \Delta_Ib_j < \Delta_IA_jA_I^{-1} b_I  + \Delta_I(n\Delta^{\max})^2
        \end{array}
        \right\}.
    \end{align*}
     Cramer's Rule implies that $\Delta_I A_jA_I^{-1} b_I  \in \Z$ for all $j \in \{1,\ldots,m\} \setminus I$. 
     	For each feasible $I$, we have $A_I^{-1} b_I \in \Po(A,b)$ and $A_jA_I^{-1} b_I \le b_j$ for all $j \in \{1, \ldots, m\}$.
     Thus,
\begin{align}\label{asymptoticContainment}
      & 
        \left\{b\in \Z^m: 
        \begin{array}{l}
        \exists \text{ a feasible basis } I \text{ and } j \in \{1,\ldots,m\}\setminus I\\[.05 cm]
        \text{with } \Delta_Ib_j < \Delta_IA_jA_I^{-1} b_I  + \Delta_I(n\Delta^{\max})^2
        \end{array}
        \right\}.\nonumber\\[.15 cm]
        \subseteq&  
        \bigcup_{\substack{I \subseteq \{1,\ldots,m\} \\ I \text{ basis}}}  \bigcup_{j \not \in I} \bigcup_{r = 0}^{\Delta_I(n\Delta^{\max})^2-1}\{b \in \Z^m : \Delta_I b_j = \Delta_IA_jA_I^{-1} b_I + r\}.
        \end{align}

Consider a feasible basis $I$ and a value $r \in \{0, \ldots, \Delta_I(n\Delta^{\max})^2-1\}$.
For each $b \in \{-t, \ldots, t\}^m$ and $j \not \in I$, the equation $\Delta_I b_j = \Delta_IA_jA_I^{-1} b_I + r$ determines $b_j$. 
Hence, 
\[
\big|\{b \in \{-t, \ldots, t\}^m : \Delta_I b_j = \Delta_IA_jA_I^{-1} b_I + r\}\big| \le \prod_{\substack{i=1 \\ i\neq j}}^m |\{-t, \ldots, t\}| = (2t+1)^{m-1}.
\] 
 Using this along with inclusion~\eqref{asymptoticContainment}, we see that
 \begin{equation}\label{eqconvergenceRate}
 \big|\{-t, \ldots, t\}^m \cap (\Z^m \setminus \cG)\big| \le {m\choose n} \cdot (m-n)\cdot n^2(\Delta^{\max})^3 \cdot (2t+1)^{m-1}.
 \end{equation}
 Hence, $ \big|\{-t, \ldots, t\}^m \cap (\Z^m \setminus \cG)\big| \in O(t^{m-1})$.
 \qed
        \end{proof}

By Lemma~\ref{Lemma:BadPointsOnHyperplanes} and the definition of $\Pr(\Z^m \setminus \cG)$, we see that $\Pr(\Z^m \setminus \cG) =0 $. 
Furthermore, Inequality~\eqref{eqconvergenceRate} implies that the $\liminf$ defining $\Pr(\Z^m \setminus \cG) $ is a limit that approaches zero at a rate of $O(1/t)$.
Hence, $\Pr(\cG) = 1 - \Pr(\Z^m \setminus \cG) = 1$. 
Consequently, almost all integer programs can be solved in polynomial time using mixed integer relaxations, provided that $\Delta$ is constant.

\begin{theorem}\label{thmAsymptotic}
The set $\cG$ satisfies $\Pr(\cG) = 1$.
If $b \in \cG$, then either $\Po(A,b) = \emptyset$ or we can identify a point $\Po(A,b) \cap \Z^n$ by finding a feasible basis $I$ for $\Po(A,b)$, computing a mixed integer hull $\WMIP(A_I, b_I)$ equal to $\IP(A_I, b_I)$, and finding a vertex of $\WMIP(A_I, b_I)$.
%
\end{theorem}

{In the algorithm outlined in Theorem~\ref{thmAsymptotic}, we test feasibility of $\IP(A_I, b_I)$ using mixed integer relaxations.
This feasibility can also be tested efficiently, at least when $\Delta$ is fixed, using previously established dynamic programs. 
One dynamic program is given by Gomory~\cite[Page 264]{G1965}. 
In his dynamic program, Gomory uses the group structure on $\Z^n$ induced by the columns of $A_I$, and the order of this group is bounded by $\Delta$.
A second dynamic program, which is found in~\cite{AEGOVW2016}, applies to non-degenerate constraints matrices, e.g., the matrix $A_I$.
A third dynamic program, given in~\cite{EW2018}, efficiently tests feasibility of $\{x \in \Z^n : A'x \le b', x \ge 0\}$ for matrices $A' \in \Z^{m'\times n}$ with the value of $m'$ considered as fixed.
Note that if $\IP(A_I, b_I)$ is transformed into Hermite Normal Form, then the corresponding value of $m'$ is bounded by $\log_2(\Delta)$.}
Although there are other methods for testing feasibility of $\IP(A_I, b_I)$, our approach is the first that tests this using mixed integer relaxations.
%

\section*{Acknowledgements}
The authors wish to thank Helene Wei{\ss} and Stefan Weltge for their help that led to major improvements of the manuscript.
We are also grateful to the anonymous referees for their comments that improved the presentation of the material.
The third author acknowledges the support from the Einstein Foundation Berlin.
 
%
\bibliographystyle{splncs04}
\bibliography{references}


\end{document}